\documentclass[a4paper,12pt]{article}
\usepackage{amsmath,amsthm,amssymb}

\numberwithin{equation}{section}

\newtheorem{theorem}{Theorem}
\newtheorem{lemma}{Lemma}
\newtheorem{proposition}{Proposition}

\theoremstyle{definition}
\newtheorem{definition}{Definition}
\newtheorem{remark}{Remark}

\newcommand{\pd}{\partial}
\newcommand{\C}{\mathbb{C}}
\newcommand{\R}{\mathbb{R}}

\newcommand{\Z}{\mathbb{Z}}
\newcommand{\T}{\mathbb{T}}
\newcommand{\dual}[2]{\langle #1, #2\rangle}
\newcommand{\eps}{\varepsilon}

\begin{document}

\begin{center}
\Large{\textbf{
Instability of solitary waves 
for nonlinear Schr\"odinger equations of derivative type
}}
\end{center}

\begin{center}
Dedicated to Professor Nakao Hayashi on his sixtieth birthday
\end{center}

\vspace{5mm}

\begin{center}
{\large Masahito Ohta} 
\end{center}
\begin{center}
Department of Mathematics, 
Tokyo University of Science, \\
1-3 Kagurazaka, Shinjuku-ku, Tokyo 162-8601, Japan \\
E-mail: mohta@rs.tus.ac.jp
\end{center}

\begin{abstract}
We study the orbital stablity and instability of solitary wave solutions 
for nonlinear Schr\"odinger equations of derivative type. 
\end{abstract}

\section{Introduction}

In this paper, we study the instability of solitary wave solutions 
for nonlinear Schr\"odinger equations of the form
\begin{equation}\label{dnls}
i\partial_tu
=-\pd_x^2u-i|u|^2 \pd_x u-b|u|^4u,
\quad (t,x)\in \R\times \R,
\end{equation}
where $b\ge 0$ is a constant. 
Eq. \eqref{dnls} appears in various areas of physics such as 
plasma physics, nonlinear optics, and so on 
(see, e.g., \cite{phys1,phys2} and also Introduction of \cite{oza}). 
It is known that 
\eqref{dnls} has a two parameter family of solitary wave solutions 
\begin{equation}\label{solitary}
u_{\omega}(t,x)
=e^{i\omega_0 t}\phi_{\omega}(x-\omega_1 t),
\end{equation}
where $\omega=(\omega_0,\omega_1)\in 
\Omega:=\{(\omega_0,\omega_1)\in \R^2: \omega_1^2<4 \omega_0\}$, 
$\gamma=1+\dfrac{16}{3}b$, 
\begin{align}
&\phi_{\omega}(x)=\tilde{\phi}_{\omega}(x) \exp \left(
i\frac{\omega_1}{2} x-\frac{i}{4} \int_{-\infty}^{x}|\tilde{\phi}_{\omega}(\eta)|^2\,d\eta\right), 
\label{solitary2} \\
&\tilde{\phi}_{\omega}(x)
=\left\{\frac{2(4\omega_0-\omega_1^2)}
{-\omega_1+\sqrt{\omega_1^2+\gamma (4\omega_0-\omega_1^2)}\,
\cosh (\sqrt{4\omega_0-\omega_1^2}\,x)}\right\}^{1/2}. 
\label{solitary3}
\end{align}
Here, we note that $\phi_{\omega}(x)$ is a solution of 
\begin{equation}\label{sp1}
-\pd_x^2\phi+\omega_0 \phi+\omega_1 i\pd_x \phi-i |\phi|^2 \pd_x \phi-b |\phi|^4\phi=0, 
\quad x\in \R,
\end{equation}
and $\tilde{\phi}_{\omega}(x)$ is a solution of 
\begin{equation}\label{sp2}
-\pd_x^2 \phi+\frac{4\omega_0-\omega_1^2}{4}\phi
+\frac{\omega_1}{2}|\phi|^2\phi
-\frac{3}{16}\gamma  |\phi|^4\phi=0,
\quad x\in \R.
\end{equation}

For $v$, $w\in L^2(\R)=L^2(\R,\C)$, we define
$$(v,w)_{L^2}=\Re \int_{\R} v(x) \overline{w(x)}\,dx,$$
and regard $L^2(\R)$ as a real Hilbert space. 
Similarly, $H^1(\R)=H^1(\R,\C)$ is regarded as a real Hilbert space with inner product 
$$(v,w)_{H^1}=(v,w)_{L^2}+(\pd_x v, \pd_x w)_{L^2}.$$

We define the energy $E:H^1(\R)\to \R$ by 
\begin{equation}\label{energy}
E(v)=\frac{1}{2}\|\pd_x v\|_{L^2}^2
-\frac{1}{4} (i|v|^2 \pd_x v, v)_{L^2}
-\frac{b}{6}\|v\|_{L^6}^6. 
\end{equation}
Then, we have 
$$E'(v)=-\pd_x^2v-i|v|^2 \pd_x v-b|v|^4v,$$
and \eqref{dnls} can be written in a Hamiltonian form 
$i\pd_t u=E'(u)$ in $H^{-1}(\R)$. 

For $\theta=(\theta_0,\theta_1)\in \R^2$ and $v\in H^1(\R)$, we define 
\begin{equation}\label{T0}
T(\theta)v(x)=e^{i\theta_0}v(x-\theta_1) \quad (x\in \R).
\end{equation}
Note that the energy $E$ is invariant under $T$, i.e., 
\begin{equation}\label{T1}
E(T(\theta)v)=E(v), \quad \theta\in \R^2, \, v\in H^1(\R),
\end{equation}
and that the solitary wave solution \eqref{solitary} 
is written as $u_{\omega}(t)=T(\omega t)\phi_{\omega}$. 

The Cauchy problem for \eqref{dnls} is locally well-posed in the energy space $H^1(\R)$ 
(see \cite{oza} and also \cite{hay,hay-oza1,hay-oza2}). 
For any $u_0\in H^1(\R)$, there exist $T_{\max}\in (0,\infty]$ and a unique solution 
$u\in C([0,T_{\max}),H^1(\R))$ of \eqref{dnls} with $u(0)=u_0$ 
such that either $T_{\max}=\infty$ or 
$T_{\max}<\infty$ and $\displaystyle{\lim_{t\to T_{\max}}\|u(t)\|_{H^1}=\infty}$. 
Moreover, the solution $u(t)$ satisfies 
\begin{align*}
E(u(t))=E(u_0), \quad Q_0(u(t))=Q_0(u_0), \quad Q_1(u(t))=Q_1(u_0)
\end{align*}
for all $t\in [0,T_{\max})$, 
where $Q_0$ and $Q_1$ are defined by 
\begin{equation}\label{charge}
Q_0(v)=\frac{1}{2}\|v\|_{L^2}^2, \quad
Q_1(v)=\frac{1}{2} (i\pd_x v,v)_{L^2}. 
\end{equation}

For $\eps>0$, we define
$$U_{\eps}(\phi_{\omega})=\{u\in H^1(\R): 
\inf_{\theta\in \R^2}\|u-T(\theta)\phi_{\omega}\|_{H^1}<\varepsilon\}.$$
Then, the stability and instability of solitary waves are defined as follows. 

\begin{definition}\label{def1}
We say that the solitary wave solution $T(\omega t)\phi_{\omega}$ 
of \eqref{dnls} is {\em stable} if 
for any $\varepsilon>0$ there exists $\delta>0$ such that 
if $u_0\in U_{\delta}(\phi_{\omega})$, 
then the solution $u(t)$ of \eqref{dnls} with $u(0)=u_0$ exists for all $t\ge 0$,  
and $u(t)\in U_{\eps}(\phi_{\omega})$ for all $t\ge 0$. 
Otherwise, $T(\omega t)\phi_{\omega}$ is said to be {\em unstable}. 
\end{definition}

For the case $b=0$, 
Colin and Ohta \cite{CO} proved that 
the solitary wave solution $T(\omega t)\phi_{\omega}$ of \eqref{dnls} 
is stable for all $\omega\in \Omega$ 
(see also \cite{GW,ZQZG}). 
We remark that the instability of solitary waves for \eqref{dnls} is not studied 
in previous papers \cite{CO,GW,ZQZG}. 
For a recent result on a generalized derivative nonlinear Schr\"odinger equation, 
see \cite{LSS}. 

In this paper, we consider the case $b>0$, and prove the following. 

\begin{theorem}\label{thm1}
Let $b>0$. 
Then there exists $\kappa=\kappa (b)\in (0,1)$ such that 
the solitary wave solution $T(\omega t)\phi_{\omega}$ of \eqref{dnls} 
is stable if $-2\sqrt{\omega_0}<\omega_1<2\kappa\sqrt{\omega_0}$, 
and unstable if $2\kappa\sqrt{\omega_0}<\omega_1<2\sqrt{\omega_0}$. 
\end{theorem}

\begin{remark}\label{rem1}
Let $b>0$, $\gamma=1+\dfrac{16}{3}b$, and 
\begin{equation}\label{def:g}
g(\xi)=\frac{2(\gamma-1)}{\xi}
\tan^{-1} \frac{1+\sqrt{1+\xi^2}}{\xi}, \quad \xi\in (0,\infty).
\end{equation}
Then, $g:(0,\infty)\to (0,\infty)$ is strictly decreasing and bijective. 
Thus, for any $b>0$, 
there exists a unique $\hat \xi=\hat \xi(b)\in (0,\infty)$ such that $g(\hat \xi)=1$. 
The constant $\kappa$ in Theorem \ref{thm1} 
is given by $\kappa=(1+\hat \xi^2/\gamma)^{-1/2}$
(see Lemma \ref{thm4} below). 
\end{remark}

\begin{remark}
The sufficient condition $-2\sqrt{\omega_0}<\omega_1<2\kappa\sqrt{\omega_0}$ 
for stability of $T(\omega t)\phi_{\omega}$ is equivalent to $Q_1(\phi_{\omega})>0$, 
and the sufficient condition $2\kappa\sqrt{\omega_0}<\omega_1<2\sqrt{\omega_0}$  
for instability is equivalent to $Q_1(\phi_{\omega})<0$ 
(see Lemma \ref{thm4} and Proof of Theorem \ref{thm1} below). 
We also remark that 
$E(\phi_{\omega})=-\dfrac{\omega_1}{2}Q_1(\phi_{\omega})$ 
for all $\omega\in \Omega$. 
\end{remark}

\begin{remark}
We do not study the borderline case $\omega_1=2\kappa \sqrt{\omega_0}$ in this paper, 
and leave it as an open problem. 
Note that $E(\phi_{\omega})=Q_1(\phi_{\omega})=0$ 
in the case $\omega_1=2\kappa \sqrt{\omega_0}$. 
For related results for one-parameter family of solitary waves in borderline cases, 
see \cite{CP,OT,oht2,maeda}. 
\end{remark}

\begin{remark}
It is not known whether \eqref{dnls} has finite time blowup solutions or not. 
It will be interesting to study relations 
between unstable solitary wave solutions obtained in Theorem \ref{thm1} 
and the existence of blowup solutions for \eqref{dnls}. 
For a recent progress in this direction, see Wu \cite{wu1,wu2}. 
\end{remark}

For $\omega\in \Omega$, 
we define the action $S_{\omega}:H^1(\R)\to \R$ by 
$$S_{\omega}(v)=E(v)+\sum_{j=0}^{1}\omega_j Q_j(v),$$
where $E$, $Q_0$ and $Q_1$ are defined by \eqref{energy} and \eqref{charge}. 
Note that $Q_0'(v)=v$, $Q_1'(v)=i\pd_xv$, 
and that \eqref{sp1} is equivalent to $S_{\omega}'(\phi)=0$. 

We also define a function $d:\Omega\to \R$ by 
$$d(\omega)=S_{\omega}(\phi_{\omega})
=E(\phi_{\omega})
+\sum_{j=0}^{1}\omega_j Q_j(\phi_{\omega}).$$
Then, we have 
$$d'(\omega)=(\pd_{\omega_0}d(\omega), \pd_{\omega_1}d(\omega))
=(Q_0(\phi_{\omega}),Q_1(\phi_{\omega})),$$
and the Hessian matrix $d''(\omega)$ of $d(\omega)$ is given by 
$$d''(\omega)
=\left[\begin{array}{cc}
\pd_{\omega_0}^2 d(\omega) &   \pd_{\omega_1} \pd_{\omega_0} d(\omega) \\
\pd_{\omega_0} \pd_{\omega_1} d(\omega) & \pd_{\omega_1}^2 d(\omega) 
\end{array}\right]
=\left[\begin{array}{rr}
\pd_{\omega_0} Q_0(\phi_{\omega}) &   \pd_{\omega_1} Q_0(\phi_{\omega}) \\
\pd_{\omega_0} Q_1(\phi_{\omega}) &   \pd_{\omega_1} Q_1(\phi_{\omega})
\end{array}\right].$$

To prove Theorem \ref{thm1}, we use the following sufficient conditions 
for stability and instability in terms of the Hessian matrix $d''(\omega)$ 
(see \cite{gss2}). 

\begin{theorem}\label{thm2}
Let $\omega\in \Omega$. 
If the matrix $d''(\omega)$ has a positive eigenvalue, 
then the solitary wave solution 
$T(\omega t)\phi_{\omega}$ of \eqref{dnls} is stable. 
\end{theorem}

\begin{theorem}\label{thm3}
Let $\omega\in \Omega$. 
If $d''(\omega)$ is negative definite 
(all eigenvalues of $d''(\omega)$ are negative), 
then the solitary wave solution 
$T(\omega t)\phi_{\omega}$ of \eqref{dnls} is unstable. 
\end{theorem}

Theorem \ref{thm2} can be proved in the same way as in Colin and Ohta \cite{CO}, 
and we omit the proof. 
We give the proof of Theorem \ref{thm3} in Section \ref{sect3} below. 
As we stated above, the instability of solitary waves for \eqref{dnls} 
has not been studied in previous papers \cite{CO,GW,ZQZG}. 

Moreover, by the explicit form \eqref{solitary2} with \eqref{solitary3} of $\phi_{\omega}$, 
and by elementary computations, 
we have the following. 

\begin{lemma}\label{thm4}
Let $b>0$ and $\gamma=1+\dfrac{16}{3}b$. 
For $\omega\in \Omega$,  we have 
\begin{align*}
&Q_0(\phi_{\omega})
=\frac{4}{\sqrt{\gamma}}
\tan^{-1}
\frac{\omega_1+\sqrt{\omega_1^2+\gamma (4\omega_0-\omega_1^2)}}
{\sqrt{\gamma (4\omega_0-\omega_1^2)}}, \\
&Q_1(\phi_{\omega})
=\frac{1}{\gamma^{3/2}}\left\{\sqrt{\gamma (4\omega_0-\omega_1^2)}\right. \\
&\hspace{30mm} \left.
-2(\gamma-1) \omega_1 \tan^{-1}
\frac{\omega_1+\sqrt{\omega_1^2+\gamma (4\omega_0-\omega_1^2)}}
{\sqrt{\gamma (4\omega_0-\omega_1^2)}}\right\}, \\
&\det [d''(\omega)]
=\frac{-\gamma Q_1(\phi_{\omega})}
{\sqrt{\gamma (4\omega_0-\omega_1^2)}
\{\omega_1^2+\gamma (4\omega_0-\omega_1^2)\}}. 
\end{align*}
\end{lemma}

Theorem \ref{thm1} follows from Theorems \ref{thm2} and \ref{thm3}, 
Lemma \ref{thm4} and Remark \ref{rem1}. 

\begin{proof}[Proof of Theorem \ref{thm1}]
Let $\omega\in \Omega$. 
If $\omega_1\le 0$, then by Lemma \ref{thm4}, we have 
$Q_1(\phi_{\omega})>0$ and $\det [d''(\omega)]<0$. 
Thus, the matrix $d''(\omega)$ has one positive eigenvalue 
and one negative eigenvalue. 
Therefore, by Theorem \ref{thm2}, $T(\omega t)\phi_{\omega}$ is stable. 

Next, we consider the case $\omega_1>0$. 
We put 
$\displaystyle{\xi=\sqrt{\gamma\left(\frac{4\omega_0}{\omega_1^2}-1\right)}}$. 
Then, by Lemma \ref{thm4}, we have 
$$Q_1(\phi_{\omega})
=\dfrac{1}{\gamma} \sqrt{4\omega_0-\omega_1^2} \, \left\{1-g(\xi)\right\},$$
where $g(\xi)$ is defined by \eqref{def:g} in Remark \ref{rem1}. 

If $g(\xi)<1$, 
then $Q_1(\phi_{\omega})>0$ and $\det [d''(\omega)]<0$. 
Thus, $d''(\omega)$ has a positive eigenvalue, 
and by Theorem \ref{thm2}, $T(\omega t)\phi_{\omega}$ is stable. 

On the other hand, if $g(\xi)>1$, 
then $Q_1(\phi_{\omega})<0$ and $\det [d''(\omega)]>0$. 
Moreover, since 
\begin{align*}
\pd_{\omega_0}^2 d(\omega)
=\pd_{\omega_0} Q_0(\phi_{\omega})
=\frac{-4 \omega_1}{\sqrt{4\omega_0-\omega_1^2}
\{\gamma (4\omega_0-\omega_1^2)+\omega_1^2\}}<0, 
\end{align*}
we see that $d''(\omega)$ is negative definite. 
Thus, it follows from Theorem \ref{thm3} that $T(\omega t)\phi_{\omega}$ is unstable. 

Finally, by Remark \ref{rem1}, 
we see that $g(\xi)<1$ is equivalent to 
$\omega_1<2 \kappa \sqrt{\omega_0}$, 
and that $g(\xi)>1$ is equivalent to 
$\omega_1>2\kappa \sqrt{\omega_0}$. 
\end{proof}

The rest of the paper is organized as follows. 
In Section \ref{sect2}, we give a variational characterization of $\phi_{\omega}$. 
This part is essentially the same as Section 3 of \cite{CO}, 
so we omit the details. 
In Section \ref{sect3}, we give the proof of Theorem \ref{thm3}. 
We divide the proof into two parts. 
In Subsection \ref{sect31}, we prove that 
if $d''(\omega)$ is negative definite,  
then there exists an unstable direction $\psi$. 
In Subsection \ref{sect32}, 
we prove the instability of $T(\omega t) \phi_{\omega}$ 
using the variational characterization of $\phi_{\omega}$ 
and the unstable direction $\psi$. 

\section{Variational characterization}\label{sect2}

In this section, we give a variational characterization of $\phi_{\omega}$. 
Although $\phi_{\omega}$ is given by \eqref{solitary2} and \eqref{solitary3} explicitly, 
we need such a variational characterization to prove stability and instability 
of solitary wave solutions $T(\omega t) \phi_{\omega}$. 

Throughout this section, we assume that $b>0$. 
The case $b=0$ is studied in Section 3 of \cite{CO}, 
and the proof for the case $b>0$ is almost the same as that for $b=0$, 
so we will omit the details. 

For $\omega\in \Omega$, we define 
\begin{align*}
&L_{\omega}(v)
=\|\pd_x v\|_{L^2}^2+\omega_0 \|v\|_{L^2}^2+\omega_1 (i\pd_x v,v)_{L^2}, \\ 
&S_{\omega} (v)
=\frac{1}{2}L_{\omega}(v)
-\frac{1}{4}(i|v|^2\pd_x v,v)_{L^2}-\frac{b}{6}\|v\|_{L^6}^6, \\
&K_{\omega}(v)
=L_{\omega}(v)-(i|v|^2\pd_x v,v)_{L^2}-b\|v\|_{L^6}^6, 
\end{align*}
and consider the following minimization problem:
\begin{equation}\label{mini1}
\mu(\omega)=\inf \{S_{\omega}(v): v\in H^1(\R)\setminus \{0\},~ K_{\omega}(v)=0\}. 
\end{equation}
Note that \eqref{sp1} is equivalent to $S_{\omega}'(\phi)=0$ 
and that $K_{\omega}(v)=\pd_{\lambda} S_{\omega} (\lambda v) |_{\lambda=1}$. 

We also define 
\begin{align*}
\tilde S_{\omega} (v)=S_{\omega} (v)-\frac{1}{4}K_{\omega}(v) 
=\frac{1}{4} L_{\omega}(v)+\frac{b}{12}\|v\|_{L^6}^6. 
\end{align*}

\begin{lemma}
Let $\omega\in \Omega$. 
\begin{description}
\item[{\rm (1)}]\
There exists a constant $C_1=C_1(\omega)>0$ such that 
$L_{\omega}(v)\ge C_1\|v\|_{H^1}^2$ for all $v\in H^1(\R)$. 
\item[{\rm (2)}]\
$\mu (\omega)>0$. 
\item[{\rm (3)}]\
If $v\in H^1(\R)$ satisfies $K_{\omega}(v)<0$, 
then $\mu (\omega)<\tilde S_{\omega}(v)$. 
\end{description}
\end{lemma}

\begin{proof}
(1) \hspace{1mm} 
See Lemma 7 (1) of \cite{CO}. 

\noindent (2) \hspace{1mm} 
Let $v\in H^1(\R)\setminus \{0\}$ satisfy $K_{\omega}(v)=0$. 
Then, by (1) and the Sobolev inequality, 
there exists $C_2>0$ such that 
\begin{align*}
&C_1 \|v\|_{H^1}^2\le L_{\omega}(v)
=(i|v|^2\pd_x v,v)_{L^2}+b\|v\|_{L^6}^6 \\
&\le \|\pd_x v\|_{L^2}\|v\|_{L^6}^3+b\|v\|_{L^6}^6
\le \frac{C_1}{2}\|v\|_{H^1}^2+C_2 \|v\|_{H^1}^6.
\end{align*}
Since $v\ne 0$, we have $\|v\|_{H^1}^4\ge \frac{C_1}{2C_2}$. 
Thus, we have 
\begin{align*}
&\mu (\omega)
=\inf \{\tilde S_{\omega}(v): v\in H^1(\R)\setminus \{0\},~ K_{\omega}(v)=0\} \\
&\ge \frac{1}{4}\inf \{L_{\omega}(v): v\in H^1(\R)\setminus \{0\},~ K_{\omega}(v)=0\}
\ge \frac{C_1}{4}\sqrt{\frac{C_1}{2C_2}}>0. 
\end{align*}

\noindent (3) \hspace{1mm} 
Let $v\in H^1(\R)\setminus \{0\}$ satisfy $K_{\omega}(v)<0$. 
Then, there exists $\lambda_1\in (0,1)$ such that 
$$K_{\omega}(\lambda_1 v)
=\lambda_1^2 L_{\omega}(v)
-\lambda_1^4 (i|v|^2\pd_x v,v)_{L^2}-\lambda_1^6 b\|v\|_{L^6}^6=0.$$
Since $v\ne 0$, we have 
$$\mu (\omega)\le \tilde S_{\omega} (\lambda_1 v)
=\frac{\lambda_1^2}{4}L_{\omega}(v)+\frac{\lambda_1^6b}{12}\|v\|_{L^6}^6
<\tilde S_{\omega} (v).$$
This completes the proof. 
\end{proof}

Let $\mathcal{M}_{\omega}$ be the set of all minimizers for \eqref{mini1}, i.e., 
\begin{align*}
\mathcal{M}_{\omega}
=\{\varphi \in H^1(\R)\setminus \{0\}: 
S_{\omega}(\varphi)=\mu (\omega),~ K_{\omega}(\varphi)=0\}. 
\end{align*}
Then, we obtain the following. 

\begin{lemma} \label{lem:vc1}
For any $\omega\in \Omega$, 
we have $\mathcal{M}_{\omega}
=\{T(\theta) \phi_{\omega}: \theta\in \R^2\}$. 
In particular, 
if $v\in H^1(\R)$ satisfies $K_{\omega}(v)=0$ and $v\ne 0$, 
then $S_{\omega}(\phi_{\omega})\le S_{\omega}(v)$. 
\end{lemma}

The proof of Lemma \ref{lem:vc1} 
is almost the same as that of Lemma 10 of \cite{CO}, 
so we omit it. 

The following lemma plays an important role in the proof of Lemma \ref{lem4.2}. 

\begin{lemma} \label{lem:vc2}
If $v\in H^1(\R)$ satisfies $\dual{K_{\omega}'(\phi_{\omega})}{v}=0$, 
then $\dual{S_{\omega}''(\phi_{\omega}) v}{v}\ge 0$. 
\end{lemma}

\begin{proof}
Let $v\in H^1(\R)$ satisfy $\dual{K_{\omega}'(\phi_{\omega})}{v}=0$. 
Since $K_{\omega}(\phi_{\omega})=0$ and 
$\dual{K_{\omega}'(\phi_{\omega})}{\phi_{\omega}}\ne 0$, 
by the implicit function theorem, 
there exist a constant $\delta>0$ and 
a $C^2$-function $\gamma:(-\delta,\delta)\to \R$ such that 
$\gamma (0)=0$ and 
\begin{equation}\label{vc3}
K_{\omega}(\phi_{\omega}+s v+\gamma (s) \phi_{\omega})=0, 
\quad s\in (-\delta,\delta). 
\end{equation}
Taking $\delta$ smaller if necessary, 
we also have $\phi_{\omega}+s v+\gamma (s) \phi_{\omega}\ne 0$ 
for $s\in (-\delta,\delta)$.

Differentiating \eqref{vc3} at $s=0$, we have 
$$0=\dual{K_{\omega}'(\phi_{\omega})}{v}
+\gamma'(0) \dual{K_{\omega}'(\phi_{\omega})}{\phi_{\omega}}.$$
Since $\dual{K_{\omega}'(\phi_{\omega})}{v}=0$ and 
$\dual{K_{\omega}'(\phi_{\omega})}{\phi_{\omega}}\ne 0$, 
we have $\gamma'(0)=0$. 

Moreover, since $\phi_{\omega}\in \mathcal{M}_{\omega}$ by Lemma \ref{lem:vc1}, 
it follows from \eqref{vc3} that the function 
$s\mapsto S_{\omega}(\phi_{\omega}+s v+\gamma (s) \phi_{\omega})$ has a 
local minimum at $s=0$. 
Thus, we have 
\begin{align*}
0
&\le \frac{d^2}{ds^2} S_{\omega}(\phi_{\omega}+s v+\gamma (s) \phi_{\omega})\big|_{s=0} \\
&=\dual{S_{\omega}''(\phi_{\omega})(v+\gamma'(0) \phi_{\omega})}{v+\gamma'(0) \phi_{\omega}}
+\dual{S_{\omega}'(\phi_{\omega})}{\gamma''(0)\phi_{\omega}} \\
&=\dual{S_{\omega}''(\phi_{\omega})v}{v}. 
\end{align*}
This completes the proof. 
\end{proof}

\section{Proof of Theorem \ref{thm3}}\label{sect3}

In this section, we give the proof of Theorem \ref{thm3}. 
We divide the proof into two parts. 
In Subsection \ref{sect31}, we prove that 
if $d''(\omega)$ is negative definite,  
then there exists an unstable direction $\psi$ 
(see Lemma \ref{lem2.3}). 
In Subsection \ref{sect32}, 
we prove the instability of $T(\omega t) \phi_{\omega}$ 
using the variational characterization of $\phi_{\omega}$ 
and the unstable direction $\psi$ (see Proposition \ref{prop1}). 
Theorem \ref{thm3} follows from Lemma \ref{lem2.3} and Proposition \ref{prop1}. 

\subsection{Existence of unstable direction}\label{sect31}

\begin{lemma}\label{lem2.2}
$\dual{S_{\omega}''(\phi_{\omega})\phi_{\omega}}{\phi_{\omega}}<0$. 
\end{lemma}

\begin{proof}
Since the function 
\begin{align*}
(0,\infty)\ni \lambda \mapsto S_{\omega}(\lambda \phi_{\omega})
=\frac{\lambda^2}{2} L_{\omega}(\phi_{\omega})
-\frac{\lambda^4}{4}(i|\phi_{\omega}|^2\pd_x \phi_{\omega},\phi_{\omega})_{L^2}
-\frac{\lambda^6\,b}{6}\|\phi_{\omega}\|_{L^6}^6
\end{align*}
has a strictly local maximum at $\lambda=1$, 
we have 
\begin{align*}
0>\frac{d^2}{d \lambda^2} S_{\omega}(\lambda \phi_{\omega})\big|_{\lambda=1} 
=\dual{S_{\omega}''(\phi_{\omega})\phi_{\omega}}{\phi_{\omega}}. 
\end{align*}
This completes the proof. 
\end{proof}

\begin{lemma}\label{lem2.3}
Assume that $d''(\hat \omega)$ is negative definite. 
Then there exists $\psi\in H^1(\R)$ such that 
$$\dual{Q_0'(\phi_{\hat \omega})}{\psi}=\dual{Q_1'(\phi_{\hat \omega})}{\psi}=0,
\quad \dual{S_{\hat \omega}''(\phi_{\hat \omega})\psi}{\psi}<0.$$
\end{lemma}

\begin{proof}
For $(s,\omega)$ near $(0, \hat \omega)$ in $\R\times \Omega$, 
we define 
$$F(s,\omega):=
\left[\begin{array}{r}
Q_0(s \phi_{\hat \omega}+\phi_{\omega})-Q_0(\phi_{\hat \omega}) \\
Q_1(s \phi_{\hat \omega}+\phi_{\omega})-Q_1(\phi_{\hat \omega})
\end{array}\right].$$
Then, we have $F(0, \hat \omega)=0$. 
Moreover, since $D_{\omega}F(0,\hat \omega)=d''(\hat \omega)$ 
is negative definite and invertible, 
by the implicit function theorem, there exist 
a constant $\delta>0$ and a $C^1$-function 
$\gamma:(-\delta,\delta)\to \Omega$ such that $\gamma(0)=\hat \omega$ and 
$$Q_0(s \phi_{\hat \omega}+\phi_{\gamma(s)})=Q_0(\phi_{\hat \omega}), \quad 
Q_1(s \phi_{\hat \omega}+\phi_{\gamma(s)})=Q_1(\phi_{\hat \omega})$$
for $s\in (-\delta,\delta)$. 
We define $\varphi_s:=s \phi_{\hat \omega}+\phi_{\gamma(s)}$ for $s\in (-\delta,\delta)$, and 
$$w_j:=\pd_{\omega_j}\phi_{\omega}|_{\omega=\hat \omega} \quad (j=0,1), \quad 
\psi:=\pd_s \varphi_s|_{s=0}=\phi_{\hat \omega}+\sum_{j=0}^{1}\gamma_j'(0) w_j.$$
Then, for $j=0,1$, we have 
\begin{align}
0&=\frac{d}{ds}Q_j(\varphi_s)|_{s=0}
=\dual{Q_j'(\phi_{\hat \omega})}{\psi} 
\label{Qs1} \\
&=\dual{Q_j'(\phi_{\hat \omega})}{\phi_{\hat \omega}}
+\sum_{k=0}^{1}\gamma_k'(0)
\dual{Q_j'(\phi_{\hat \omega})}{w_k}. 
\nonumber
\end{align}
Moreover, differentiating
$$0=S_{\omega}'(\phi_{\omega})
=E'(\phi_{\omega})+\sum_{k=0}^{1}\omega_k Q_k'(\phi_{\omega}),$$
with respect to $\omega_j$ for $j=0,1$, we have  
\begin{align}
0
&=E''(\phi_{\omega})(\pd_{\omega_j} \phi_{\omega})
+\sum_{k=0}^{1}\omega_k Q_k''(\phi_{\omega})(\pd_{\omega_j} \phi_{\omega})
+Q_j'(\phi_{\omega}) 
\label{Qs2} \\
&=S_{\omega}''(\phi_{\omega})(\pd_{\omega_j} \phi_{\omega})+Q_j'(\phi_{\omega}). 
\nonumber 
\end{align}
By \eqref{Qs1} and \eqref{Qs2}, we have 
\begin{align*}
&\dual{S_{\hat \omega}''(\phi_{\hat \omega})\psi}{\psi} 
=\dual{S_{\hat \omega}''(\phi_{\hat \omega})\phi_{\hat \omega}}{\phi_{\hat \omega}} 
+2\sum_{j=0}^{1} \gamma_j'(0) 
\dual{S_{\hat \omega}''(\phi_{\hat \omega})w_j}{\phi_{\hat \omega}} \\ 
&\hspace{65mm}
+\sum_{j,k=0}^{1} \gamma_j'(0) \gamma_k'(0)
\dual{S_{\hat \omega}''(\phi_{\hat \omega})w_j}{w_k} \\
&=\dual{S_{\hat \omega}''(\phi_{\hat \omega})\phi_{\hat \omega}}{\phi_{\hat \omega}}
-2\sum_{j=0}^{1} \gamma_j'(0) \dual{Q_j'(\phi_{\hat \omega})}{\phi_{\hat \omega}} 
-\sum_{j,k=0}^{1} \gamma_j'(0) \gamma_k'(0)
\dual{Q_j'(\phi_{\hat \omega})}{w_k} \\
&=\dual{S_{\hat \omega}''(\phi_{\hat \omega})\phi_{\hat \omega}}{\phi_{\hat \omega}}
+\sum_{j,k=0}^{1} \gamma_j'(0) \gamma_k'(0)
\dual{Q_j'(\phi_{\hat \omega})}{w_k} \\
&=\dual{S_{\hat \omega}''(\phi_{\hat \omega})\phi_{\hat \omega}}{\phi_{\hat \omega}}
+\sum_{j,k=0}^{1} \gamma_j'(0) \gamma_k'(0)\,
\pd_{\omega_j} \pd_{\omega_k} d(\hat \omega).
\end{align*}
Since $d''(\hat \omega)$ is negative definite, it follows from Lemma \ref{lem2.2} that  
$$\dual{S_{\hat \omega}''(\phi_{\hat \omega})\psi}{\psi}
\le \dual{S_{\hat \omega}''(\phi_{\hat \omega})\phi_{\hat \omega}}{\phi_{\hat \omega}}<0.$$
This completes the proof. 
\end{proof}

\subsection{Proof of instability}\label{sect32}

In this subsection, we prove the following. 

\begin{proposition}\label{prop1}
Let $\omega\in \Omega$, 
and assume that there exists $\psi\in H^1(\R)$ such that 
\begin{equation}\label{ass:psi}
\dual{Q_0'(\phi_{\omega})}{\psi}=\dual{Q_1'(\phi_{\omega})}{\psi}=0,
\quad \dual{S_{\omega}''(\phi_{\omega})\psi}{\psi}<0.
\end{equation}
Then, the solitary wave solution $T(\omega t)\phi_{\omega}$ of \eqref{dnls} is unstable. 
\end{proposition}

To prove Proposition \ref{prop1}, 
we use the argument of Gon\c{c}alves Ribeiro \cite{GR} 
(see also \cite{sha-str,gss1}) with some modifications. 
Throughout this subsection, 
we fix $\omega\in \Omega$, 
and assume that $\psi\in H^1(\R)$ satisfies \eqref{ass:psi}. 

\begin{lemma}\label{lem3.1}
There exists a constant $\lambda_0>0$ such that 
\begin{align*}
S_{\omega}(\phi_{\omega}+\lambda \psi)<S_{\omega}(\phi_{\omega})
\end{align*}
for all $\lambda\in (-\lambda_0,0)\cup (0,\lambda_0)$. 
\end{lemma}

\begin{proof}
By Taylor's expansion, for $\lambda\in \R$, we have 
\begin{align*}
&S_{\omega}(\phi_{\omega}+\lambda \psi) \\
&=S_{\omega}(\phi_{\omega})+\lambda \dual{S_{\omega}'(\phi_{\omega})}{\psi}
+\lambda^2 \int_{0}^{1}(1-s)\dual{S_{\omega}''(\phi_{\omega}+s\lambda \psi)\psi}{\psi}\,ds \\
&=S_{\omega}(\phi_{\omega})
+\lambda^2 \int_{0}^{1}(1-s)\dual{S_{\omega}''(\phi_{\omega}+s\lambda \psi)\psi}{\psi}\,ds.
\end{align*}
Since $\dual{S_{\omega}''(\phi_{\omega})\psi}{\psi}<0$, 
by the continuity of $\lambda\mapsto \dual{S_{\omega}''(\phi_{\omega}+\lambda \psi)\psi}{\psi}$, 
there exists $\lambda_0>0$ such that 
$$\dual{S_{\omega}''(\phi_{\omega}+\lambda \psi)\psi}{\psi}
\le \frac{1}{2}\dual{S_{\omega}''(\phi_{\omega})\psi}{\psi}$$
for all $\lambda \in (-\lambda_0,\lambda_0)$. 
Thus, for $\lambda\in (-\lambda_0,0)\cup (0,\lambda_0)$, we have 
\begin{align*}
S_{\omega}(\phi_{\omega}+\lambda \psi)
\le S_{\omega}(\phi_{\omega})
+\frac{\lambda^2}{4} \dual{S_{\omega}''(\phi_{\omega})\psi}{\psi}
<S_{\omega}(\phi_{\omega}). 
\end{align*}
This completes the proof. 
\end{proof}

For $u\in H^1(\R)$, we define 
\begin{align*}
T_0'\, u=iu, \quad T_1'\, u=-\pd_x u.
\end{align*}
Then, by \eqref{T0} and \eqref{charge}, 
we have 
\begin{equation} \label{QT1}
\pd_{\theta_j}T(\theta) u=T(\theta)T_j'\, u=T_j'\, T(\theta)u, \quad 
\dual{Q_j'(u)}{v}=(T_j'\, u,iv)_{L^2}
\end{equation}
for $\theta=(\theta_0,\theta_1)\in \R^2$, $u$, $v\in H^1(\R)$ and $j=0,1$. 
We denote $\T=\R/2\pi \Z$. 

\begin{lemma}\label{lem3.3}
There exist a constant $\varepsilon_0>0$ and a $C^1$-function 
$$\alpha=(\alpha_0,\alpha_1):U_{\varepsilon_0}(\phi_{\omega})\to \T\times \R$$ 
such that $\alpha (\phi_{\omega})=0$, and 
\begin{description}
\item[{\rm (1)}]\
$\alpha (T(\xi)u)=\alpha (u)+\xi$
for all $u\in U_{\eps_0}(\phi_{\omega})$ and $\xi \in \T\times \R$. 
\item[{\rm (2)}]\
$(T_j'\, u,T(\alpha (u))\phi_{\omega})_{L^2}=0$ 
for all $u\in U_{\eps_0}(\phi_{\omega})$ and $j=0$, $1$. 
\item[{\rm (3)}]\
There exists $\rho>0$ such that 
$$\sum_{j,k=0}^{1}(T_j'\, u,T(\alpha (u))T_k'\, \phi_{\omega})_{L^2}
\zeta_j\zeta_k \ge \rho |\zeta|^2$$
for all $u\in U_{\eps_0}(\phi_{\omega})$ and $\zeta=(\zeta_0,\zeta_1) \in \R^2$.
\end{description}
\end{lemma}

\begin{proof}
See Section 3 of \cite{GR}. 
\end{proof}

For $u\in U_{\eps_0}(\phi_{\omega})$, 
we define 
$$H(u)=[h_{jk}(u)]_{j,k=0,1}, \quad 
h_{jk}(u)=(T_j'\, u,T(\alpha (u))T_k'\, \phi_{\omega})_{L^2}.$$
Then, by Lemma \ref{lem3.3} (1), we have  
\begin{equation}\label{hjk1}
h_{jk}(T(\xi)u)
=(T(\xi)T_j'\, u,T(\alpha (u)+\xi)T_k'\, \phi_{\omega})_{L^2}=h_{jk}(u)
\end{equation}
for $u\in U_{\eps_0}(\phi_{\omega})$ and $\xi \in \T\times \R$. 

Moreover, differentiating Lemma \ref{lem3.3} (2) with respect to $u$, 
we have 
\begin{equation}\label{hjk2}
\sum_{k=0}^{1}h_{jk}(u) \dual{\alpha_k'(u)}{w}
=(T(\alpha (u))T_j'\, \phi_{\omega},w)_{L^2}
\end{equation}
for $u\in U_{\eps_0}(\phi_{\omega})$, $w\in H^1(\R)$ and $j=0,1$. 
By Lemma \ref{lem3.3} (3), the matrix $H(u)$ is invertible, 
and we denote the inverse $H(u)^{-1}$ by $G(u)=[g_{jk}(u)]$. 
Then, there exists a constant $C>0$ such that 
\begin{equation}\label{gjk1}
|g_{jk}(u)|\le C \hspace{2mm}\mbox{for all} \hspace{2mm} 
u\in U_{\eps_0}(\phi_{\omega}), ~ j,k=0,1.
\end{equation}

For $j=0,1$ and $u\in U_{\eps_0}(\phi_{\omega})$, we define 
$$a_j(u):=\sum_{k=0}^{1}g_{jk}(u)T(\alpha (u))T_k'\, \phi_{\omega}.$$
Since $\phi_{\omega}\in H^2(\R)$, 
we see that $a_j(u) \in H^1(\R)$, 
it follows from \eqref{hjk2} that 
$$\dual{\alpha_j'(u)}{w}=(a_j(u),w)_{L^2}, \quad w\in H^1(\R).$$
By \eqref{hjk1} and Lemma \ref{lem3.3} (1), 
for  $j=0,1$, we have 
\begin{equation}\label{aj1}
a_j(T(\xi)u)=T(\xi)a_j(u)
\hspace{2mm}\mbox{for all} \hspace{2mm} 
u\in U_{\eps_0}(\phi_{\omega}),~ \xi \in \T\times \R. 
\end{equation}
Moreover, by \eqref{gjk1}, there exists a constant $C>0$ such that 
\begin{equation}\label{aj2}
\|a_j(u)\|_{H^1}\le C \hspace{2mm}\mbox{for all} \hspace{2mm} 
u\in U_{\eps_0}(\phi_{\omega}), ~ j=0,1.
\end{equation}

Next, for $u\in U_{\eps_0}(\phi_{\omega})$, we define 
\begin{align}
&A(u)=\left(iu,T(\alpha(u))\psi \right)_{L^2}, \label{def:Au}\\
&q(u)=T(\alpha(u))\psi
+\sum_{j=0}^{1}\left(iu,T(\alpha(u))T_j'\, \psi \right)_{L^2} i a_j(u).
\label{def:qu}
\end{align}
Then, since $\psi$, $a_0(u)$, $a_1(u)\in H^1(\R)$, 
we see that $q(u)\in H^1(\R)$. 

\begin{lemma}\label{lem3.4}
For $u\in U_{\eps_0}(\phi_{\omega})$, 
\begin{description}
\item[{\rm (1)}]\
$A(T(\xi)u)=A(u)$, $q(T(\xi)u)=T(\xi)q(u)$ 
for all $\xi \in \T\times \R$. 
\item[{\rm (2)}]\
$\dual{A'(u)}{w}=(q(u),iw)_{L^2}$ for $w\in H^1(\R)$. 
\item[{\rm (3)}]\
$q(\phi_{\omega})=\psi$. 
 \item[{\rm (4)}]\
$\dual{Q_j'(u)}{q(u)}=0$ for $j=0,1$. 
\end{description}
\end{lemma}

\begin{proof}
\par \noindent (1) \hspace{1mm}
By Lemma \ref{lem3.3} (1), we have 
\begin{align*}
A(T(\xi)u)
&=\left(iT(\xi)u,T(\alpha(u)+\xi)\psi \right)_{L^2} \\
&=\left(iT(\xi)u,T(\xi)T(\alpha(u))\psi \right)_{L^2}=A(u).
\end{align*}
Moreover, by \eqref{aj1}, we have 
\begin{align*}
q(T(\xi)u)
&=T(\xi)T(\alpha(u))\psi
+\sum_{j=0}^{1} \left(iT(\xi)u,T(\xi)T(\alpha(u))T_j'\, \psi \right)_{L^2} i a_j(T(\xi)u) \\
&=T(\xi) q(u).
\end{align*}
\par \noindent (2) \hspace{1mm}
For $u\in U_{\eps_0}(\phi_{\omega})$ and $w\in H^1(\R)$, 
we have 
\begin{align*}
\dual{A'(u)}{w}
&=\left(iw,T(\alpha(u))\psi \right)_{L^2}
+\sum_{j=0}^{1}\dual{\alpha_j'(u)}{w} \left(iu,T(\alpha(u))T_j'\, \psi \right)_{L^2} \\
&=\left(iw,T(\alpha(u))\psi \right)_{L^2}
+\sum_{j=0}^{1} \left(iu,T(\alpha(u))T_j'\, \psi \right)_{L^2} \left(a_j(u),w \right)_{L^2} \\
&=\left(q(u), i w \right)_{L^2}. 
\end{align*}
\par \noindent (3) \hspace{1mm}
By \eqref{QT1} and the assumption \eqref{ass:psi}, we have 
$$(i\phi_{\omega},T_j'\, \psi)_{L^2}
=(T_j'\, \phi_{\omega},i\psi)_{L^2}
=\dual{Q_j'(\phi_{\omega})}{\psi}=0.$$
Moreover, since $\alpha(\phi_{\omega})=0$, 
by \eqref{def:qu}, we have $q(\phi_{\omega})=\psi$. 
\par \noindent (4) \hspace{1mm}
For $u\in H^2(\R)\cap U_{\eps_0}(\phi_{\omega})$, 
by (1) and (2), we have 
\begin{align*}
0=\pd_{\xi_j} A(T(\xi)u) \big|_{\xi=0}
=\dual{A'(u)}{T_j'\, u}
=(q(u),iT_j'\, u)_{L^2}.
\end{align*}
By density argument, we have 
$(q(u),iT_j'\, u)_{L^2}=0$ for all $u\in U_{\eps_0}(\phi_{\omega})$. 

Thus, we have $\dual{Q_j'(u)}{q(u)}=(T_j'\, u,iq(u))_{L^2}=0$ 
for $u\in U_{\eps_0}(\phi_{\omega})$. 
\end{proof}

For $u\in U_{\varepsilon_0}(\phi_{\omega})$, we define
$$P(u):=\dual{E'(u)}{q(u)}.$$
We remark that by Lemma \ref{lem3.4} (4), we have 
\begin{equation}\label{P1}
P(u)=\dual{S_{\omega}'(u)}{q(u)}, \quad 
u\in U_{\varepsilon_0}(\phi_{\omega}).
\end{equation}

\begin{lemma}\label{lem:AP}
Let $I$ be an interval of $\R$. 
Let $u\in C(I,H^1(\R))\cap C^1(I,H^{-1}(\R))$ be a solution of \eqref{dnls}, 
and assume that $u(t)\in U_{\varepsilon_0}(\phi_{\omega})$ for all $t\in I$. 
Then,
$$\frac{d}{dt}A(u(t))=P(u(t))$$
for all $t\in I$. 
\end{lemma}

\begin{proof}
By Lemma 4.6 of \cite{gss1} 
and Lemma \ref{lem3.4} (2), 
we see that $t\mapsto A(u(t))$ is a $C^1$-function on $I$, and 
\begin{align*}
\frac{d}{dt}A(u(t))
=\dual{i\pd_t u(t)}{q(u(t))}
\end{align*}
for all $t\in I$. 
Since $u(t)$ is a solution of \eqref{dnls}, we have 
\begin{align*}
\dual{i\pd_t u(t)}{q(u(t))}
=\dual{E'(u(t))}{q(u(t))}=P(u(t))
\end{align*}
for all $t\in I$. 
This completes the proof. 
\end{proof}

\begin{lemma}\label{lem4.1}
There exist constants 
$\lambda_1>0$ and $\eps_1\in (0,\eps_0)$ such that 
\begin{equation*}
S_{\omega}(u+\lambda q(u))\le S_{\omega}(u)+\lambda P(u)
\end{equation*}
for all $\lambda\in (-\lambda_1,\lambda_1)$ and $u\in U_{\eps_1}(\phi_{\omega})$. 
\end{lemma}

\begin{proof}
For $u\in U_{\eps_0}(\phi_{\omega})$ and $\lambda\in \R$, 
by Taylor's expansion, we have 
\begin{equation}\label{So1}
S_{\omega}(u+\lambda q(u))=S_{\omega}(u)+\lambda P(u)
+\lambda^2 \int_{0}^{1} (1-s) R(\lambda s,u)\,ds,
\end{equation}
where we used \eqref{P1} and put  
$$R(\lambda,u):=\dual{S_{\omega}''(u+\lambda q(u))q(u)}{q(u)}.$$
Here, we remark that 
\begin{align*}
&P(T(\xi)u)=\dual{S_{\omega}'(T(\xi)u)}{T(\xi)q(u)}=P(u), \\
&R(\lambda,T(\xi)u)=\dual{S_{\omega}''(T(\xi)(u+\lambda q(u)))T(\xi)q(u)}{T(\xi)q(u)}
=R(\lambda,u)
\end{align*}
for $\xi \in  \T\times \R$, $\lambda\in \R$ and $u\in H^1(\R)$. 
Moreover, since 
$$R(0,\phi_{\omega})
=\dual{S_{\omega}''(\phi_{\omega})q(\phi_{\omega})}{q(\phi_{\omega})}
=\dual{S_{\omega}''(\phi_{\omega})\psi}{\psi}<0,$$
by the continuity of $R(\lambda,u)$ with respect to $\lambda$ and $u$, 
there exist constants $\lambda_1>0$ and $\eps_1\in (0,\eps_0)$ 
such that $R(\lambda,u)<0$ 
for all $\lambda\in (-\lambda_1,\lambda_1)$ and $u\in U_{\eps_1}(\phi_{\omega})$. 
Thus, by \eqref{So1}, we have 
\begin{equation*}
S_{\omega}(u+\lambda q(u))\le S_{\omega}(u)+\lambda P(u)
\end{equation*}
for all $\lambda\in (-\lambda_1,\lambda_1)$ and $u\in U_{\eps_1}(\phi_{\omega})$. 
\end{proof}

\begin{lemma}\label{lem4.2}
There exist constants $\eps_2\in (0,\eps_1)$ and $\lambda_2\in (0,\lambda_1)$ 
that satisfy the following. 
For any $u\in U_{\eps_2}(\phi_{\omega})$, 
there exists $\Lambda (u) \in (-\lambda_2,\lambda_2)$ such that 
\begin{align*}
K_{\omega} \left(u+\Lambda (u) q(u)\right)=0, \quad 
u+\Lambda (u) q(u)\ne 0. 
\end{align*}
\end{lemma}

\begin{proof}
First, since $\dual{S_{\omega}''(\phi_{\omega})\psi}{\psi}<0$, 
by Lemma \ref{lem:vc2}, we have $\dual{K_{\omega}'(\phi_{\omega})}{\psi}\ne 0$. 
Thus, without loss of generality, we may assume that 
$\dual{K_{\omega}'(\phi_{\omega})}{\psi}>0$. 

For $u\in U_{\eps_0}(\phi_{\omega})$ and $\lambda\in \R$, we have 
\begin{equation}\label{Ko1}
K_{\omega}(u+\lambda q(u))
=K_{\omega}(u)+\lambda \int_{0}^{1} \dual{K_{\omega}'(u+s\lambda q(u))}{q(u)}\,ds. 
\end{equation}
Since $\dual{K_{\omega}'(\phi_{\omega})}{q(\phi_{\omega})}
=\dual{K_{\omega}'(\phi_{\omega})}{\psi}>0$, 
by the continuity of the function $\dual{K_{\omega}'(u+\lambda q(u))}{q(u)}$ 
with respect to $\lambda$ and $u$, 
there exist constants $\lambda_2\in (0,\lambda_1)$ and $\eps_2\in (0,\eps_1)$ such that 
\begin{equation}\label{Ko2}
\dual{K_{\omega}'(u+\lambda q(u))}{q(u)}
\ge \frac{1}{2}\dual{K_{\omega}'(\phi_{\omega})}{\psi}
\end{equation}
for all $\lambda \in [-\lambda_2,\lambda_2]$ and $u\in U_{\eps_2}(\phi_{\omega})$. 
Moreover, since $K_{\omega}(\phi_{\omega})=0$, 
taking $\eps_2$ smaller if necessary, we have  
\begin{equation}\label{Ko3}
|K_{\omega}(u)|<\frac{\lambda_2}{2}\dual{K_{\omega}'(\phi_{\omega})}{\psi}, \quad 
u\in U_{\eps_2}(\phi_{\omega}). 
\end{equation}

Let $u\in U_{\eps_2}(\phi_{\omega})$. 
If $K_{\omega}(u)<0$, then it follows from \eqref{Ko1}--\eqref{Ko3} that 
\begin{align*}
K_{\omega}(u+\lambda_2 q(u))
&=K_{\omega}(u)+\lambda_2 \int_{0}^{1} \dual{K_{\omega}'(u+s\lambda_2 q(u))}{q(u)}\,ds \\
&>-\frac{\lambda_2}{2}\dual{K_{\omega}'(\phi_{\omega})}{\psi}
+\frac{\lambda_2}{2}\dual{K_{\omega}'(\phi_{\omega})}{\psi}=0.
\end{align*}
Since the function $\lambda\mapsto K_{\omega}(u+\lambda q(u))$ is continuous, 
there exists $\Lambda (u)\in (0,\lambda_2)$ such that 
\begin{equation}\label{Ko4}
K_{\omega}(u+\Lambda (u) q(u))=0.
\end{equation}
Similarly, if $K_{\omega}(u)>0$, then we have 
\begin{align*}
K_{\omega}(u-\lambda_2 q(u))
&=K_{\omega}(u)-\lambda_2 \int_{0}^{1} \dual{K_{\omega}'(u-s\lambda_2 q(u))}{q(u)}\,ds \\
&<\frac{\lambda_2}{2}\dual{K_{\omega}'(\phi_{\omega})}{\psi}
-\frac{\lambda_2}{2}\dual{K_{\omega}'(\phi_{\omega})}{\psi}=0.
\end{align*}
Thus, there exists $\Lambda (u)\in (-\lambda_2,0)$ such that \eqref{Ko4}. 
If $K_{\omega}(u)=0$, 
taking $\Lambda (u)=0$, \eqref{Ko4} is satisfied. 

Finally, by \eqref{aj2} and \eqref{def:qu}, 
taking $\lambda_2$ and $\eps_2$ smaller if necessary, 
we have $u+\Lambda (u) q(u)\ne 0$ for all $u\in U_{\eps_2}(\phi_{\omega})$. 
This completes the proof. 
\end{proof}

\begin{lemma}\label{lem4.3}
Let $\lambda_2$ and $\eps_2$ be the positive constants given in Lemma \ref{lem4.2}. 
Then, 
$$S_{\omega}(\phi_{\omega})\le S_{\omega}(u)+\lambda_2 |P(u)|$$
for all $u\in U_{\eps_2}(\phi_{\omega})$. 
\end{lemma}

\begin{proof}
By Lemma \ref{lem4.2}, for any $u\in U_{\eps_2}(\phi_{\omega})$, 
there exists $\Lambda (u)\in (-\lambda_2,\lambda_2)$ such that 
$K_{\omega}(u+\Lambda (u) q(u))=0$ and $u+\Lambda (u) q(u)\ne 0$. 
Then, it follows from Lemma \ref{lem:vc1} that 
\begin{equation}\label{So3}
S_{\omega}(\phi_{\omega}) \le S_{\omega}(u+\Lambda (u) q(u)), \quad 
u\in U_{\eps_2}(\phi_{\omega}). 
\end{equation}
Thus, by Lemma \ref{lem4.1} and \eqref{So3}, 
for $u\in U_{\eps_2}(\phi_{\omega})$, we have 
\begin{align*}
S_{\omega}(\phi_{\omega})
&\le S_{\omega} \left(u+\Lambda (u) q(u) \right)
\le S_{\omega}(u)+\Lambda (u) P(u) \\
&\le S_{\omega}(u)+|\Lambda (u)| |P(u)|
\le S_{\omega}(u)+\lambda_2 |P(u)|. 
\end{align*}
This completes the proof. 
\end{proof}

We are now in a position to give the Proof of Proposition \ref{prop1}. 

\begin{proof}[Proof of Proposition \ref{prop1}]
Suppose that $T(\omega t)\phi_{\omega}$ is stable. 
For $\lambda$ close to $0$, 
let $u_{\lambda}(t)$ be the solution of \eqref{dnls} 
with $u_{\lambda}(0)=\phi_{\omega}+\lambda \psi$. 
Since $T(\omega t)\phi_{\omega}$ is stable, 
there exists $\lambda_3\in (0,\lambda_0)$ such that 
if $|\lambda|<\lambda_3$, 
then $u_{\lambda}(t)\in U_{\eps_2}(\phi_{\omega})$
for all $t\ge 0$. 
Moreover, by the definition \eqref{def:Au} of $A$, 
there exists $C_1>0$ such that 
$|A(v)|\le C_1$ for all $v\in U_{\eps_2}(\phi_{\omega})$.  

Let $\lambda\in (-\lambda_3,0)\cup (0,\lambda_3)$. 
Then, by Lemma \ref{lem3.1}, we have 
$$\delta_{\lambda}:=S_{\omega}(\phi_{\omega})-S_{\omega}(u_{\lambda}(0))>0.$$
Moreover, by Lemma \ref{lem4.3} and the conservation of $S_{\omega}$, we have 
$$0<\delta_{\lambda}=S_{\omega}(\phi_{\omega})-S_{\omega}(u_{\lambda}(t))
\le \lambda_2 |P(u_{\lambda}(t))|, \quad t\ge 0.$$
Since $t\mapsto P(u_{\lambda}(t))$ is continuous, 
we see that either 
(i) $P(u_{\lambda}(t))\ge \delta_{\lambda}/\lambda_2$ for all $t\ge 0$, or 
(ii) $P(u_{\lambda}(t))\le -\delta_{\lambda}/\lambda_2$ for all $t\ge 0$. 
Moreover, by Lemma \ref{lem:AP}, we have 
$$\frac{d}{dt} A(u_{\lambda}(t))=P(u_{\lambda}(t)), \quad t\ge 0.$$

Therefore, we see that $A(u_{\lambda}(t))\to \infty$ as $t\to \infty$ for the case (i), 
and $A(u_{\lambda}(t))\to -\infty$ as $t\to \infty$  for the case (ii). 
This contradicts the fact that 
$|A(u_{\lambda}(t))|\le C_1$ for all $t\ge 0$. 
Hence, $T(\omega t)\phi_{\omega}$ is unstable. 
\end{proof}

\vspace{2mm} \noindent 
{\bf Acknowledgment}. 
This work was supported by JSPS KAKENHI Grant Number 24540163.


\begin{thebibliography}{99}

\bibitem{CP}A.~Comech and D.~Pelinovsky, 
{\em Purely nonlinear instability of standing waves with minimal energy}, 
{Comm. Pure Appl. Math.} 
{\bf 56} (2003), 1565--1607. 

\bibitem{CO} M.~Colin and M.~Ohta, 
{\em Stability of solitary waves for derivative nonlinear Schr\"odinger equation}, 
{Ann. Inst. H. Poincar\'e, Anal. Non Lin\'eaire}
{\bf 23} (2006), 753--764. 

\bibitem{GR}J.~M.~Gon\c{c}alves Ribeiro, 
{\em Instability symmetric stationary states for some nonlinear Schr\"odinger equations 
with an external magnetic field}, 
{Ann. Inst. H. Poincar\'e, Phys. Th\'eor.} 
{\bf 54} (1991), 403--433.

\bibitem{gss1}M.~Grillakis, J.~Shatah and W.~Strauss, 
{\em Stability theory of solitary waves in the presence of symmetry, I}, 
{J. Funct. Anal.}  
{\bf 74} (1987), 160--197. 

\bibitem{gss2}M.~Grillakis, J.~Shatah and W.~Strauss, 
{\em Stability theory of solitary waves in the presence of symmetry, II}, 
{J. Funct. Anal.}  
{\bf 94} (1990), 308--348. 

\bibitem{GW}B.~Guo and Y.~Wu, 
{\em Orbital stability of solitary waves for the nonlinear derivative Schr\"odinger equation}, 
{J. Differential Equations} 
{\bf 123} (1995), 35--55. 

\bibitem{hay}N.~Hayashi, 
{\em The initial value problem for the derivative nonlinear 
Schr\"odinger equation in the energy space}, 
{Nonlinear Anal.} 
{\bf 20} (1993), 823--833. 

\bibitem{hay-oza1}N.~Hayashi and T.~Ozawa, 
{\em On the derivative nonlinear Schr\"odinger equation}, 
{Phys. D} {\bf 55} (1992), 14--36. 

\bibitem{hay-oza2}N.~Hayashi and T.~Ozawa, 
{\em Finite energy solutions of nonlinear Schr\"odinger equations of derivative type}, 
{SIAM J. Math. Anal.} 
{\bf 25} (1994), 1488--1503. 

\bibitem{LSS}X.~Liu, G.~Simpson and C.~Sulem, 
{\em Stability of solitary waves for a generalized derivative nonlinear Schr\"odinger equation}, 
{J. Nonlinear Sci.} 
{\bf 23} (2013), 557--583. 

\bibitem{maeda}M.~Maeda, 
{\em Stability of bound states of Hamiltonian PDEs in the degenerate cases}, 
{J. Funct. Anal.}  
{\bf 263} (2012), 511--528. 

\bibitem{phys1}W.~Mio, T.~Ogino, K.~Minami and S.~Takeda, 
{\em Modified nonlinear Schr\"odinger equation for Alfv\'en waves 
propagating along the magnetic field in cold plasmas}, 
{J. Phys. Soc. Japan}  
{\bf 41} (1976), 265--271.

\bibitem{phys2}E.~Mj{\o}lhus, 
{\em On the modulational instability of hydromagnetic waves 
parallel to the magnetic field}, 
{J. Plasma Phys.} 
{\bf 16} (1976), 321--334. 

\bibitem{oht2}M.~Ohta, 
{\em Instability of bound states for abstract nonlinear Schr\"odinger equations}, 
{J. Funct. Anal.}  
{\bf 261} (2011), 90--110. 

\bibitem{OT}M.~Ohta and G.~Todorova, 
{\em Strong instability of standing waves for the nonlinear Klein-Gordon equation 
and the Klein-Gordon-Zakharov system}, 
{SIAM J. Math. Anal.} 
{\bf 38} (2007), 1912--1931. 

\bibitem{oza}T.~Ozawa, 
{\em On the nonlinear Schr\"odinger equations of derivative type}, 
{Indiana Univ. Math. J.} 
{\bf 45} (1996), 137--163. 

\bibitem{sha-str}J.~Shatah and W.~Strauss,
{\em Instability of nonlinear bound states}, 
{Comm. Math. Phys.} 
{\bf 100} (1985), 173--190.

\bibitem{wu1}Y.~Wu, 
{\em Global well-posedness of the derivative nonlinear Schr\"odinger equations 
in energy space}, 
{Anal. PDE} 
{\bf 6} (2013), 1989--2002. 

\bibitem{wu2}Y.~Wu, 
{\em Global well-posedness on the derivative nonlinear Schr\"odinger equation}, 
preprint, arXiv:1404.5159.

\bibitem{ZQZG}W.~Zhang, Y.~Qin, Y.~Zhao and B.~Guo, 
{\em Orbital stability of solitary waves for Kundu equation}, 
{J. Differential Equations} 
{\bf 247} (2009), 1591--1615. 

\end{thebibliography}
\end{document}